\documentclass[11pt,a4paper]{article}
\usepackage[utf8]{inputenc}
\usepackage[a4paper, total={6.6in, 9.7in}]{geometry}

\usepackage[T1]{fontenc}
\usepackage[variablett]{lmodern}
\usepackage{amsmath}
\usepackage{amsfonts}
\usepackage{amssymb}
\usepackage{bbm}
\usepackage{float}
\usepackage{wrapfig}
\usepackage{enumerate}
\usepackage{enumitem}
\usepackage[svgnames]{xcolor}
\usepackage{graphicx}
\usepackage{quiver}
\usepackage{array}
\usepackage{longtable}
\usepackage{mathtools}

\usepackage{caption}
\usepackage[bottom]{footmisc}

\usepackage[hidelinks]{hyperref}
\definecolor{mycolor}{HTML}{750000}
\hypersetup{
    colorlinks = true,
    allcolors = mycolor,
    bookmarksdepth = 2,
}
\PassOptionsToPackage{unicode}{hyperref}
\PassOptionsToPackage{naturalnames}{hyperref}

\usepackage[nameinlink]{cleveref}
\Crefname{subsection}{Subsection}{Subsections}
\Crefname{question}{Question}{Questions}
\Crefname{subsubsection}{Paragraph}{Paragraphs}

\usepackage{comment}

\newcommand{\customref}[2]{\hyperref[#2]{#1}}

\newcommand{\andd}{\hspace{1cm}\text{and}\hspace{1cm}}

\renewcommand{\emptyset}{\varnothing}
\renewcommand{\epsilon}{\varepsilon}
\renewcommand{\phi}{\varphi}

\renewcommand{\tilde}{\widetilde}

\newcommand{\free}{\textnormal{-free}}
\newcommand{\mmod}{\textnormal{-mod}}
\newcommand{\rroot}{\textnormal{-root}}
\newcommand{\alg}{\textnormal{-alg}}
\newcommand{\Vect}{\textnormal{-Vect}}

\newcommand{\N}{\mathbb{N}}
\newcommand{\Z}{\mathbb{Z}}

\newcommand{\Q}{\mathbb{Q}}

\newcommand{\R}{\mathbb{R}}

\newcommand{\Aut}{\mathrm{Aut}} 
\newcommand{\GL}{\mathrm{GL}} 

\newcommand{\card}[1]{\leftl| #1 \rightr|} 

\newcommand{\id}{\mathrm{id}} 

\newcommand{\verti}{\ \middle \vert\ }

\renewcommand{\O}{\mathcal{O}} 

\newcommand{\simto}{\overset\sim\to}

\newcommand{\igl}{{\smallint}}

\usepackage{amsthm}
\usepackage{thmtools}

\newcounter{mycounter}[section]

\theoremstyle{plain}
\newtheorem{theorem}[mycounter]{Theorem}
\newtheorem{corollary}[mycounter]{Corollary}
\newtheorem{proposition}[mycounter]{Proposition}
\newtheorem{lemma}[mycounter]{Lemma}

\theoremstyle{remark}
\newtheorem{remark}[mycounter]{Remark}

\theoremstyle{definition}

\newtheorem{question}[mycounter]{Question}

\counterwithin{equation}{section}

\newcommand{\leftl}{\mathopen{}\mathclose\bgroup\left}
\newcommand{\rightr}{\aftergroup\egroup\right}

\usepackage{titletoc}
\titlecontents{section}
[0pt]                                               
{}%
{\contentsmargin{0pt}                               
    \large\thecontentslabel.\enspace
    }
{\contentsmargin{0pt}\large}                        
{\titlerule*[.5pc]{ }\contentspage}                 
[]                                                  

\usepackage{titlesec}
\titleformat{\section}[block]{\normalfont\centering\scshape\large}{\thesection.}{1em}{}
\titleformat{\subsection}[block]{\normalfont\large\bf}{\thesubsection.}{1em}{\bf}
\titleformat{\subsubsection}[runin]{\normalfont}{\bf\thesubsubsection.}{0.3em}{\bf}

\makeatletter
\patchcmd{\@maketitle}{\LARGE}{\huge}{\typeout{OK 1}}{\typeout{Failed 1}}
\patchcmd{\@maketitle}{\large \lineskip}{\Large \lineskip}{\typeout{OK 2}}{\typeout{Failed 2}}
\makeatother

\title{
  Symmetries of various sets of polynomials
}

\author{
  Béranger Seguin%
  \footnote{
    Universität Paderborn, Fakultät EIM, Institut für Mathematik, Warburger Str. 100, 33098 Paderborn, Germany.
    Email: \texttt{bseguin@math.upb.de}.
  }
}
\date{}

\renewenvironment{abstract}{%
\par\noindent\rule{\textwidth}{1pt}
\par\noindent\textsc{Abstract.}}
{\par\noindent\rule{\textwidth}{1pt}}

\begin{document}

\maketitle{}

\begin{abstract}
  Let $K$ be a field of characteristic~$0$, and let $k \geq 2$ be an integer.
  We prove that every $K$-linear bijection $f \colon K[X] \to K[X]$ strongly preserving the set of $k$-free polynomials (or the set of polynomials with a $k$-fold root in $K$) is a constant multiple of a $K$-algebra automorphism of~$K[X]$, i.e., that there are elements $a, c \in K^{\times}$ and $b \in K$ such that $f(P)(X) = c P(a X + b)$.
  When $K$ is a number field or $K=\R$, we prove that similar statements hold when~$f$ preserves the set of polynomials with a root in~$K$.
  
  \bigskip


  \par\noindent
	\textbf{Mathematics Subject Classification (MSC 2020):} 47B49 $\cdot$ 15A86 $\cdot$ 12E05 $\cdot$ 26C10 $\cdot$ 12E25
\end{abstract}

{
  \hypersetup{linkcolor=black}
  \tableofcontents{}
}
\par\noindent\rule{\textwidth}{1pt}

\section{Introduction}

\subsection{Introduction to linear preserver problems}

An element of a ring $B$ is \emph{$k$-free} if it does not belong to $m^k$ for any maximal ideal $m$ of~$B$.
In \cite[Theorem~1.1]{gunklu}, Gundlach and Klüners show that if $K$ is a number field and $f \colon \O_K \to \O_K$ is a $\Z$-linear map preserving the set of $k$-free elements of $\O_K$ for some $k \geq 2$, then $f$ is of the form $x \mapsto c \cdot \sigma(x)$ where $c \in \O_K^{\times}$ is a unit and $\sigma \in \Aut(K)$ is a field automorphism of $K$.
In other words, $f$ is, up to a constant factor, an automorphism of the $\Z$-algebra $\O_K$.
The motivation for that question comes from the study of dynamical systems and shift spaces associated to the set of $k$-free integers in a number field.

Our goal is to examine similar questions in different contexts.
A general setting for these questions is the following:
\begin{question}
  \label{qn-general}
  Let $A$ be a ring and $B$ be an $A$-algebra.
  Let $S$ be a subset of $B$ invariant under multiplication by units and under $A$-algebra automorphisms of~$B$, i.e., invariant under the action of~$B^{\times} \rtimes \Aut_{A\alg}(B)$.
  Let $\Aut_{A \mmod}^S (B)$ be the set of $A$-module automorphisms~$f \colon B \simto B$ such that $f(S) = S$ (we say that $f$ \emph{preserves} $S$).
  When do $\Aut_{A \mmod}^S (B)$ and $B^{\times} \rtimes \Aut_{A\alg}(B)$ coincide?
\end{question}


\Cref{qn-general} is similar to the family of ``linear preserver problems'' considered in matrix and operator theory \cite{lintrans,lpp,lppintro,presdiag}.
An example is the following result of Dieudonné~\cite{dieudonne}: let~$K$ be a field and~$f$ be an endomorphism of the space of $n \times n$ matrices over~$K$, such that a matrix~$M$ is invertible if and only if $f(M)$ is invertible; then, there are matrices~$U,V \in \GL_n (K)$ such that~$f$ is either of the form $M \mapsto U M V$, or of the form $M \mapsto U M^T V$.
The answer to \Cref{qn-general} is then negative: up to multiplication by $V U^{-1}$, the former maps do correspond to $K$-algebra automorphisms, but the latter correspond to antiautomorphisms.

In situations where \Cref{qn-general} does admit a positive answer, we can interpret this as saying that the additive structure of $B$ as an $A$-module together with the subset $S$ are ``enough information'' to reconstruct the multiplicative structure of $B$ up to a scaling factor.
Typically, \Cref{qn-general} has a negative answer when the set $S$ can be defined in terms of the additive structure of $B$ only (for example, if~$S$ is $\emptyset$, $\{0\}$, $B$ or $B \setminus \{0\}$), as then every $A$-module automorphism of~$B$ preserves~$S$.
Instead, we should consider sets $S$ defined in terms of the multiplicative structure of~$B$, e.g., the sets of units, squarefree or $k$-free elements, irreducible elements, etc.

\subsection{Main results}

In this article, we study instances of \Cref{qn-general} where $A$ is a field and $B$ is a ring of polynomials in one variable over (an extension of) $A$.
Our results are positive answers in three situations:
\begin{itemize}
  \item
    (\Cref{thm-kfree-kroot})
    $K$ is a field of characteristic zero, $k \geq 2$ is an integer, and $S \subseteq K[X]$ is either the set of $k$-free polynomials, or the set of polynomials with a $k$-fold root in $K$.
  \item
    (\Cref{thm-root-nf})
    $K$ is a number field and $S \subseteq K[X]$ is the set of polynomials with a root in $K$.
  \item
    (\Cref{thm-root-real})
    $K = \R$ and $S \subseteq \R[X]$ is the set of polynomials with a real root.
\end{itemize}
In these three cases, we show that every $K$-linear bijection $f \colon K[X] \to K[X]$ such that $f(S) = S$ is of the form $P \mapsto c \cdot P \circ (a X + b)$ for constants $a, c \in K^{\times}$ and $b \in K$.

In fact, \Cref{thm-root-nf} is stronger than the statement above, as we consider more generally the case of a finite separable extension $L|K$ of Hilbertian fields and describe every $K$-linear bijection $L[X] \to L[X]$ preserving the set of polynomials with a root in $L$.

\paragraph{Relation with previous work.}
Linear preserver problems for polynomials have been considered in \cite{fisk,gutshap,melamud}, mostly focusing on real polynomials and the location of their roots, and interpreting ``$f$ preserves $S$'' in the weaker sense $f(S) \subseteq S$.
There is also a large body of work on linear preservers of algebraic sets --- notably zeros of matrix polynomials, cf. \cite{howard,matpol}.
For polynomials of fixed degree, our results about squarefree polynomials are somewhat related to the linear preservers of the algebraic set defined by the vanishing of the discriminant.
However, the fact that we do not assume that our maps preserve degrees makes it difficult to use the language of algebraic geometry and of (finite-dimensional) algebraic varieties to think about this problem.

\bigskip

Polynomial rings are quite different from the rings of integers of number fields studied in \cite{gunklu}, as they are not finite-dimensional.
However, the fact that we obtain similar results in both cases is intriguing, and we may wonder whether these observations have a hidden common cause:

\begin{question}
  Is there a negative answer to \Cref{qn-general} satisfying the three following conditions?
  \begin{itemize}
    \item
      $A$ is a Jacobson ring and $B$ is a finitely generated commutative $A$-algebra;
    \item
      The intersection of all maximal ideals $m$ of $B$ such that $A / (m \cap A) \simeq B / m$ is the zero ideal;
    \item
      $S$ is the set of $k$-free elements of $B$ for some $k \geq 2$.
  \end{itemize}

  (These hypotheses were chosen to mimic the existence of ``enough'' primes which split completely, as they are central in the argument of \cite{gunklu}, while also including the case $A=K$, $B=K[X]$.
  See also \Cref{prop-general}.)
\end{question}

Other natural questions which are left for further study are the following:
\begin{itemize}
  \item
    Does \Cref{thm-kfree-kroot} hold in positive characteristic?
  \item
    Is there a field $K$ which is not algebraically closed, and such that there is a $K$-linear bijection $K[X] \to K[X]$, preserving the set of polynomials with a root in $K$, which is not in $K^{\times} \rtimes \Aut_{K \alg} (K[X])$?
    Settling the case of non-Archimedean local fields would be a good start.
\end{itemize}


\subsection{Notation}

In this paper, the phrase ``the map $f$ preserves the set $S$'' always means $f(S)=S$.
A \emph{$k$-fold root} of a polynomial $P$ is a root of multiplicity $\geq k$.
For a field extension $L|K$, we define the following groups:
\begin{align*}
  \Aut_{K \Vect}(L[X])
  &
  \coloneqq
  \leftl\{
    f \colon L[X] \to L[X]
    \verti
    \text{$f$ is a $K$-linear bijection}
  \rightr\}
  \\
  \Aut_{K \Vect}^{k \free} (L[X])
  &
  \coloneqq
  \leftl\{
    f \in \Aut_{K \Vect}(L[X])
    \ \verti\ 
    \forall P \in L[X], \quad
    \begin{array}{c}
      P \text{ is } k\free \\
      \Updownarrow \\
      f(P) \text{ is } k\free
    \end{array}
  \rightr\}
  \\
  \Aut_{K \Vect}^{k \rroot} (L[X])
  &
  \coloneqq
  \leftl\{
    f \in \Aut_{K \Vect}(L[X])
    \ \verti\ 
    \forall P \in L[X], \quad
    \begin{array}{c}
      P \text{ has a } k \text{-fold root in } L \\
      \Updownarrow \\
      f(P) \text{ has a } k \text{-fold root in } L
    \end{array}
  \rightr\}
  \\
  \Aut_{K \Vect}^{\textnormal{root}} (L[X])
  &
  \coloneqq
  \Aut_{K \Vect}^{1 \rroot} (L[X])
\end{align*}

\subsection{Organization of the paper}

\Cref{sn-root-permuting} consists of ``technical'' results (\Cref{thm-root-pres} and \Cref{cor-mu-linear}) which are used in the proofs of all main theorems.

In \Cref{sn-kfree-kroot}, we prove \Cref{thm-kfree-kroot}, which states that
$
  \Aut_{K \Vect}^{k \free} (K[X])
  = 
  \Aut_{K \Vect}^{k \rroot} (K[X])
  =
  K^{\times} \rtimes \Aut_{K \alg} (K[X])
$
for all fields $K$ of characteristic $0$ and all integers $k \geq 2$.
The proof is by induction on $k$, using \Cref{cor-mu-linear} as a ``replacement'' for the case $k=1$.

In \Cref{sn-root-nf}, we prove \Cref{thm-root-nf}, which states that $\Aut_{K \Vect}^{\textnormal{root}} (L[X]) = L^{\times} \rtimes \Aut_{K \alg} (L[X])$ for all extensions $L|K$ of number fields.
The proof uses Hilbert's irreducibility theorem and \Cref{thm-root-pres}.

In \Cref{sn-root-real}, we prove \Cref{thm-root-real}, which states that
$
  \Aut_{\R \Vect}^{\textnormal{root}} (\R[X])
  =
  \R^{\times} \rtimes \Aut_{\R \alg} (\R[X])
$.
The proof proceeds by showing that $\Aut_{\R \Vect}^{\textnormal{root}} (\R[X]) \subseteq \Aut_{\R \Vect}^{2 \rroot} (\R[X])$,
and applying \Cref{thm-kfree-kroot} (or the weaker \Cref{thm-sub-main}) to conclude.

The dependencies between the sections are summed up by the following diagram:
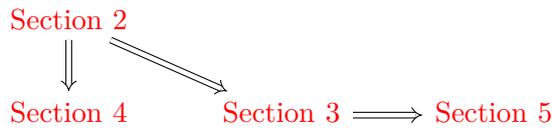
\[\begin{tikzcd}[ampersand replacement=\&]
	{\textnormal{\Cref{sn-root-permuting}}} \\
	{\textnormal{\Cref{sn-root-nf}}}
  \& {\textnormal{\Cref{sn-kfree-kroot}}}
  \& {\textnormal{\Cref{sn-root-real}}}
	\arrow[Rightarrow, from=1-1, to=2-1]
	\arrow[Rightarrow, from=1-1, to=2-2]
	\arrow[Rightarrow, from=2-2, to=2-3]
\end{tikzcd}\]

\subsection{Acknowledgements}

This work was supported by the Deutsche Forschungsgemeinschaft (DFG, German Research Foundation) --- Project-ID 491392403 --- TRR 358.
I thank Fabian Gundlach for listening to my endless discussions about this problem, and for suggesting the idea of using linear subspaces of fixed codimension to ``track'' the roots, and I thank the anonymous referee for their careful proofreading and their helpful corrections, and in particular for simplifying the final computation in the proof of \Cref{thm-sub-main}.
I also thank Philippe Caldero for making a video version of the proof of the case $k=2$ of \Cref{thm-kfree-kroot} on his YouTube channel, which I highly recommend for French speakers: \url{https://www.youtube.com/watch?v=2oF32X3UHK0}.

\section{Root-permuting maps}
\label{sn-root-permuting}

In this section, we characterize maps which preserve roots ``up to a fixed permutation of the base field'' (see \Cref{cor-mu-linear} for a precise statement).
The main results of this section are \Cref{thm-root-pres} and \Cref{cor-mu-linear}, which we use in all further sections.
We begin with the following proposition, which is stated in an abstract way as it may be used to answer \Cref{qn-general} in more general settings:

\begin{proposition}
  \label{prop-general}
  Let $A$ be a ring, let $B$, $B'$ be commutative $A$-algebras, and let $f$ be an $A$-linear map $B \to B'$.
  Denote by $M_{\deg=1}(B)$ the set of maximal ideals~$m$ of $B$ such that the canonical map $A \to B / m$ is surjective.
  For~$m \in M_{\deg=1}(B)$, let~$\mu(m)$ be the ideal of $B'$ generated by~$f(m)$, and assume that the intersection of all ideals $\mu(m)$ for~$m \in M_{\deg=1}(B)$ is the zero ideal of~$B'$.
  Then:
  \[
    \forall x,y \in B, \quad
    f(1) f(x y) = f(x) f(y).
  \]
  If moreover $B=B'$ and $f$ is bijective, then $f \in B^{\times} \rtimes \Aut_{A \alg}(B)$.
\end{proposition}

\begin{proof}
  For each $m \in M_{\deg=1}(B)$, $f$ induces a well-defined $A$-linear map $f_m \colon B / m \to B' / \mu(m)$ which fits in the following commutative diagram of $A$-linear maps:
  \[\begin{tikzcd}[ampersand replacement=\&]
    A \& B \& {B'} \\
    \& {B/m} \& {B'/\mu(m)}
    \arrow[from=1-1, to=1-2]
    \arrow[two heads, from=1-1, to=2-2]
    \arrow["f", from=1-2, to=1-3]
    \arrow[two heads, from=1-2, to=2-2]
    \arrow[two heads, from=1-3, to=2-3]
    \arrow["{f_m}"', from=2-2, to=2-3]
  \end{tikzcd}\]
  Fix $x, y \in B$.
  For each $m \in M_{\deg=1}(B)$, use surjectivity to choose elements $x_m, y_m \in A$ whose images in $B / m$ coincide respectively with those of $x$ and $y$.
  By $A$-linearity, we have $f(1) f(x_m y_m) = x_m y_m f(1)^2 = f(x_m) f(y_m)$, and thus $f_m (1)f_m (x y)$ equals $f_m (x) f_m (y)$ in $B' / \mu(m)$.
  The element $f(1) f(x y) - f(x) f(y)$ then belongs to $\bigcap_{m \in M_{\deg=1}(B)} \mu(m) = 0$ and hence $f(1) f(x y) = f(x) f(y)$.

  Assume now that~$f$ is surjective and pick~$u$ such that $f(u) = 1$.
  Then, $f(1)f(u^2)=f(u)f(u)=1$ and thus $f(1) \in B'^{\times}$.
  The $A$-linear map $g \colon B \to B'$ defined by $x \mapsto f(1)^{-1} f(x)$ satisfies $g(x y) = g(x) g(y)$ and hence is a morphism of $A$-algebras $B \to B'$.
  The last statement follows.
\end{proof}

(The proof can be summed up by the following slogan:
at primes of degree $1$, $A$-linearity implies local multiplicativity by unidimensionality; if there are enough primes of degree $1$, then the local multiplicativity at all those primes implies multiplicativity.)

\bigskip

We now characterize linear maps which preserve roots of polynomials up to a fixed permutation of the base field.
We include the case of a nontrivial algebraic field extension $L|K$ as we need it for the proof of \Cref{thm-root-nf}.
We include the possibility of a finite set $S$ of exceptional elements of $K$ because we need this generalization for the proof of \Cref{thm-sub-main}.

\begin{theorem}
  \label{thm-root-pres}
  Let $L|K$ be an algebraic field extension, with $K$ infinite.
  Let $S$ be a finite subset of~$L$, and let $L' \coloneqq L \setminus S$.
  Let $\mu \colon L' \to L$ be an injective map and $f \colon L[X] \to L[X]$ be a $K$-linear map such that $f(1)$ does not vanish on $\mu(L')$, and such that
  \[
    \forall P \in L[X], \quad
    \forall x \in L', \quad
    \bigl[
      P(x)=0
    \bigr]
    \Rightarrow
    \bigl[
      f(P)(\mu(x)) = 0
    \bigr].
  \]
  Then, there is a polynomial $Y \in L[X]$ and an automorphism $\sigma \in \Aut(L|K)$ such that
  \[
    \forall P \in L[X], \quad
    f(P) = f(1) \cdot (\sigma(P) \circ Y).
  \]
  Moreover, the maps $\sigma$ and $Y \circ \mu$ coincide on $L'$.
\end{theorem}

\begin{proof}
  Let $K' \coloneqq K \cap L'$.
  For every element $u \in L^{\times}$, let $f_u$ be the $K$-linear map $K[X] \to L[X]$ defined by $P \mapsto f(u P)$.
  By hypothesis, for any $x \in K'$, the image under $f_u$ of the maximal ideal $(X-x) \in M_{\deg=1} (K[X])$ is contained in the ideal $(X-\mu(x))$ of $L[X]$.
  An element of $\bigcap_{x \in K'} (X-\mu(x))$ is a polynomial vanishing on the set $\mu(K')$ (which is infinite as $\mu$ is injective and $K'$ is infinite) and hence is the zero polynomial.
  By \Cref{prop-general}, we have
  \begin{equation}
    \label{eqn-morphicity-f1}
    \forall u \in L, \quad
    \forall P,Q \in K[X], \quad
    f(u) f(u P Q) = f(u P) f(u Q).
  \end{equation}
  Consider arbitrary $x \in L$ and $P \in K[X]$, and apply \Cref{eqn-morphicity-f1} with $Q = P$ and $u \in \{1, x, x+1\}$.
  We obtain
  \begin{align*}
    f(x+1)f((x+1)P^2)
    &
    =
    f(x)f(x P^2) + f(x)f(P^2) + f(1)f(x P^2) + f(1)f(P^2)
    \\
    &
    =
    f(x P)^2 + f(x)f(P^2) + f(1)f(x P^2) + f(P)^2 
  \end{align*}
  but also
  \begin{align*}
    f(x+1)f((x+1)P^2)
    &
    =
    f((x+1)P)^2
    \\
    &
    =
    f(x P)^2 + 2 f(x P)f(P) + f(P)^2
  \end{align*}
  and thus
  \[
    f(x)f(P^2) + f(1)f(x P^2) = 2 f(x P)f(P).
  \]
  Multiplying both sides by $f(1)f(x)$ and applying \Cref{eqn-morphicity-f1} again yields
  \[
    f(x)^2f(P)^2 + f(1)^2f(x P)^2 = 2 f(1)f(x)f(x P)f(P),
  \]
  which rewrites as $[f(1)f(x P) - f(x) f(P)]^2 = 0$.
  We have shown
  \[
    \forall x \in L, \quad
    \forall P \in K[X], \quad
    f(1)f(x P) = f(x) f(P).
  \]
  In other words, the $K$-linear map $g \colon L[X] \to L(X)$ defined by $P \mapsto \frac{f(P)}{f(1)}$ (the polynomial $f(1)$ is nonzero as it does not vanish on $\mu(L')$) satisfies:
  \begin{equation}
    \label{eqn-quasi-morphic}
    \forall x \in L, \quad
    \forall P \in K[X], \quad
    g(x P) = g(x) g(P).
  \end{equation}
  Moreover, by \Cref{eqn-morphicity-f1} with $u=1$, the restriction of $g$ to $K[X]$ is a morphism of $K$-algebras $K[X] \to L(X)$.
  We define $Y \coloneqq g(X)$.
  Note that for any $x \in L'$, the polynomial $X-x$ vanishes at~$x$ and hence $g(X-x) = Y-g(x)$ vanishes at $\mu(x)$.

  Consider arbitrary elements $u, v \in L$.
  We are going to prove that $g(u v) = g(u) g(v)$.
  If $v = 0$, then $g(u v) = g(u)g(v) = 0$; thus, we assume that $v \in L^{\times}$.
  Let $W(v)$ be the infinite set of elements $t \in K^{\times}$ such that $t v \in L'$.
  For $t \in W(v)$, the polynomial $(X-u)(X - t v) \in L[X]$ vanishes at $t v$, and thus the polynomial $g\bigl((X-u)(X - t v)\bigr)$ vanishes at $\mu(t v)$.
  We have
  \begin{align*}
    g\bigl(
      (X-u)(X - t v)
    \bigr)
    & =
    g\leftl(
      X^2 - (u+t v)X + u t v
    \rightr)
    \\
    & =
    Y^2 - g(u+t v) Y + t g(u v)
    \\
    & =
    (Y-g(u))(Y-g(t v)) + t [g(u v) - g(u) g(v)].
  \end{align*}
  Evaluating at $\mu(t v)$, we find that $0 = 0 + t [g(u v)(\mu(t v)) - g(u)(\mu(t v)) \cdot g(v)(\mu(t v))]$, so the rational function $g(u v) - g(u) g(v) \in L(X)$ vanishes on the infinite set $\leftl\{ \mu(t v) \verti t \in W(v) \rightr\}$, and therefore it is zero.
  We have shown:
  \begin{equation}
    \label{eqn-g-mult-on-L}
    \forall u,v \in L, \quad
    g(u v) = g(u) g(v).
  \end{equation}
  Now, consider a $K$-basis $(\beta_i)_{i \in I}$ of $L$ and arbitrary polynomials $P, Q \in L[X]$, written as
  $
    P =
    \sum_{i \in I}
      \beta_i P_i
  $ and $
    Q =
    \sum_{i \in I}
      \beta_i Q_i
  $
  with $P_i, Q_i \in K[X]$.
  We have:
  \begin{align*}
    g(P Q)
    & =
    g\leftl(
      \sum_{i,j}
        \beta_i \beta_j
        P_i Q_j
    \rightr)
    \\
    & =
    \sum_{i,j}
      g( \beta_i \beta_j )
      g( P_i Q_j  )
    &
    \text{by \Cref{eqn-quasi-morphic}}
    \\
    & =
    \sum_{i,j}
      g( \beta_i ) g ( \beta_j )
      g( P_i ) g ( Q_j )
    &
    \text{by}
    \begin{cases}
      \text{\Cref{eqn-g-mult-on-L}} \\
      \text{\Cref{eqn-morphicity-f1} with $u=1$}
    \end{cases}
    \\
    & =
    \biggl(
      \sum_i
        g( \beta_i )
        g( P_i )
    \biggr)
    \biggl(
      \sum_j
        g ( \beta_j )
        g ( Q_j )
    \biggr)
    \\
    & =
    g(P)
    g(Q)
    &
    \text{by \Cref{eqn-quasi-morphic}}.
  \end{align*}
  We have shown that $g \colon L[X] \to L(X)$ is a morphism of $K$-algebras.
  As elements of $L$ are algebraic over $K$, their images in $L(X)$ are algebraic over $K$ and hence constant (i.e., they belong to~$L$); this implies that $g$ restricts to an automorphism $\sigma \in \Aut(L|K)$.
  We then have:
  \[
    \forall P \in L[X], \quad
    g(P) = \sigma(P) \circ Y.
  \]
  This implies:
  \[
    \forall P \in L[X], \quad
    f(P) = f(1) \cdot (\sigma(P) \circ Y).
  \]
  The fact that $f(X^n) = f(1) Y^n$ is a polynomial for all $n \geq 1$ implies that $Y \in L[X]$.
  It remains only to show that $\sigma$ coincides with $Y \circ \mu$ on $L'$.
  For this, notice that for any $x \in L'$, the polynomial~$X-x$ has a root at $x$, and hence $g(X - x) = Y - \sigma(x)$ has a root at $\mu(x)$ as~$\mu(x) \in \mu(L')$ is not a root of~$f(1)$.
  Evaluating at $\mu(x)$, we obtain~$Y(\mu(x))=\sigma(x)$.
\end{proof}

\begin{corollary}
  \label{cor-mu-linear}
  Let $K$ be an infinite field.
  Let $S$ be a finite subset of $K$ and $K' \coloneqq K \setminus S$.
  Let $\mu \colon K \to K$ be a bijective map and $f_1, f_2 \colon K[X] \to K[X]$ be two $K$-linear maps such that:
  \[
    \forall P \in K[X], \quad
    \forall x \in K', \quad
    P(x) = 0
    \iff f_1(P)(\mu(x)) = 0
    \iff f_2(P)(\mu^{-1}(x)) = 0.
  \]
  Then, there are constants $a \in K^{\times}$ and $b \in K$ such that:
  \begin{enumerate}
    \item
      $\mu(x) = \frac{x-b}{a}$ for all $x \in K'$
    \item
      $
        f_1(P)(X)
        =
        f_1(1)
        P(a X + b)
      $ and $
        f_2(P)(X)
        =
        f_2(1)
        P\leftl(
          \frac{X-b}{a}
        \rightr)
      $, for all $P \in K[X]$
    \item
      $f_1 \circ f_2 = f_2 \circ f_1 = f_1(1) f_2(1) \cdot \id$.
  \end{enumerate}
\end{corollary}

\begin{proof}
  Note that the hypothesis implies that $f_1(1)$ does not vanish on $\mu(K')$ and that $f_2(1)$ does not vanish on $\mu^{-1}(K')$.
  Applying \Cref{thm-root-pres} (with $L=K$) to the triples $(S, \mu|_{K'}, f_1)$ and $(S, \mu^{-1}|_{K'}, f_2)$, we deduce that there are polynomials $Y_1, Y_2 \in K[X]$ such that, for all $P \in K[X]$:
  \[
    f_1(P) = f_1(1) \cdot (P \circ Y_1)
    \andd
    f_2(P) = f_2(1) \cdot (P \circ Y_2)    
  \]
  and such that both $Y_1 \circ \mu$ and $Y_2 \circ \mu^{-1}$ coincide with the identity on $K'$.
  In particular, the polynomial maps $Y_1$ and $Y_2$ coincide respectively with the maps $\mu^{-1}$ and $\mu$ outside the finite set $\mu(S) \cup \mu^{-1}(S)$.
  Since $\mu^{-1} \circ \mu = \id$, the algebraic equality $Y_1 \circ Y_2 = X$ holds in $K[X]$, and hence both $Y_1$ and $Y_2$ have degree~$1$.
  Writing $Y_1$ as $a X + b$, we have $Y_2 = \frac{X-b}{a}$ and $\mu(x) = Y_2(x)$ for $x \in K'$, and all points follow.
\end{proof}

\begin{remark}
  When $L|K$ is a nontrivial extension, the conclusion of \Cref{thm-root-pres} could not follow directly from an application of \Cref{prop-general} as the set $M_{\deg=1}(L[X])$ is empty when $L[X]$ is seen as a $K$-algebra.
\end{remark}

\section{Symmetries of the set of $k$-free polynomials}
\label{sn-kfree-kroot}

The goal of this section is to prove the following theorem:

\begin{theorem}
  \label{thm-kfree-kroot}
  Let $K$ be a field of characteristic zero and $k \geq 2$ be an integer.
  Assume that $f \colon K[X] \to K[X]$ is a $K$-linear bijection preserving the set of $k$-free polynomials, or preserving the set of polynomials with no $k$-fold root in $K$.
  Then, there are constants $a,c \in K^{\times}$, $b \in K$ such that:
  \[
    \forall P \in K[X], \quad
    f(P) = c \cdot P \circ (a X + b).
  \]
  In other words, $\Aut_{K \Vect}^{k \free} (K[X]) = \Aut_{K \Vect}^{k \rroot} (K[X]) = K^{\times} \rtimes \Aut_{K \alg}(K[X])$.
\end{theorem}

The proof of \Cref{thm-kfree-kroot} is in two steps.
First, we establish in \Cref{subsn-kfree-to-kroot} that maps in $\Aut_{\R \Vect}^{k \rroot} (\R[X])$ or in $\Aut_{\R \Vect}^{k \free} (\R[X])$ preserve the set of polynomials with a given $k$-fold root ``up to a fixed permutation of $K$''.
In \Cref{subsn-kroot-induction}, we show that maps with that property are of the expected form, by induction on $k$, reducing to \Cref{cor-mu-linear} for the case $k=1$.
These pieces are put together in \Cref{subsn-proof-kfree-kroot} to prove \Cref{thm-kfree-kroot}.

\subsection{From non-$k$-free polynomials to $k$-fold roots}
\label{subsn-kfree-to-kroot}

We fix an integer $k \geq 2$ and an infinite field $K$ of characteristic zero or $> k$.
The main result of this subsection is \Cref{cor-dblroot-exist-mu}, which states that $K$-linear bijections preserving the set of $k$-free polynomials or the set of polynomials with no $k$-fold root in $K$ actually preserve the set of polynomials with a given $k$-fold root ``up to a fixed permutation of $K$''.
The starting point is the following lemma:

\begin{lemma}
  \label{lem-kfree-to-kroot}
  Let $x \in K$ and $P \in K[X]$.
  Let $K'$ be the set of elements $\lambda \in K$ such that $P + \lambda (X-x)^k$ is not $k$-free.
  Assume that the cardinality of $K'$ is at least $1 + \deg P$ (for example, if $K'$ is infinite).
  Then, the polynomial $(X-x)^k$ divides~$P$.
\end{lemma}

\begin{proof}
  By hypothesis, the following polynomial has positive degree for all $\lambda \in K'$:
  \begin{align*}
    g(\lambda)
    & \coloneqq
    \gcd\leftl(
      P + \lambda (X-x)^k,
      P' + \lambda k (X-x)^{k-1},
      \dots,
      P^{(k-1)} + \lambda k! (X-x)
    \rightr)
    \\
    & =
    \gcd\leftl(
      P - \frac1k (X-x) P',
      P' + \lambda k (X-x)^{k-1},
      \dots,
      P^{(k-1)} + \lambda k! (X-x)
    \rightr)
  \end{align*}
  Denote by $\tilde{P}$ the polynomial $P - \frac1k (X-x) P'$, which is independent of $\lambda$.

  First assume that $\tilde{P} = 0$, i.e., that $P = \frac1k (X-x) P'$.
  Let $R$ be the monic greatest common divisor of $P$ and $(X-x)^k$.
  We have $R = (X-x)^r$ for some integer $r \leq k$.
  If $r = k$, then $(X-x)^k$ divides~$P$ and we are done, so we assume that $r < k$.
  In particular, $k-r \in \{1, \dots, k\}$ is nonzero in~$K$.
  Write $P = (X-x)^r Q$, where $Q$ is not divisible by $(X-x)$.
  The differential equation $k P = (X-x) P'$ rewrites as $(k-r)Q = (X-x)Q'$, which implies that $(X-x)$ divides $Q$, a contradiction.

  Assume now that $\tilde{P} \neq 0$, and let $\tilde{P}_1, \dots, \tilde{P}_r$ be the monic irreducible factors of~$\tilde{P}$, of which there are at most $\deg P$.
  For each $\lambda \in K'$, $g(\lambda)$ is not constant and divides~$\tilde{P}$, and thus is a multiple of~$\tilde{P}_{i(\lambda)}$ for some $i(\lambda) \in {1,\dots,r}$.
  Then, $(\tilde{P}_{i(\lambda)})^k$ divides $P+\lambda(X-x)^k$.
  Since $\card{K'} > \deg P \geq r$, there are two distinct $\lambda_1, \lambda_2 \in K'$ such that $i(\lambda_1)=i(\lambda_2)$; we denote this common value by $i_0$.
  Then, $(\tilde{P}_{i_0})^k$ divides $P+\lambda_1 (X-x)^k$ and $P+\lambda_2(X-x)^k$, and hence divides both $P$ and $(X-x)^k$.
  Since $(X-x)^k$ has no irreducible factor besides $X-x$, we have $\tilde{P}_{i_0} = X-x$ and finally $(X-x)^k$ divides~$P$.
\end{proof}

\begin{corollary}
  \label{cor-codimk-nonkfree}
  Let $V$ be a $K$-linear subspace of $K[X]$ of codimension $k$ containing only polynomials which are not $k$-free.
  Then, there is some $x \in K$ such that $V = (X-x)^k K[X]$.
\end{corollary}

\begin{proof}
  Since $V$ has codimension $k$, the space $V_k$ of polynomials of degree $\leq k$ in~$V$ has dimension at least $1$.
  Since polynomials of degree $< k$ are $k$-free, they are not in $V$ and thus~$V_k$ has dimension exactly $1$.
  We have $V_k = K \cdot P$ for a monic polynomial $P$ of degree~$k$ which is not $k$-free, and which is therefore of the form $(X-x)^k$ for some $x \in K$.
  For all $Q \in V$ and $\lambda \in K$, the polynomial $Q + \lambda (X-x)^k \in V$ is not $k$-free and thus~$(X-x)^k$ divides~$Q$ by \Cref{lem-kfree-to-kroot}, as~$K$ is infinite.
  This establishes $V \subseteq (X-x)^k K[X]$.
  Since both spaces have codimension~$k$, they are equal.
\end{proof}

\begin{corollary}
  \label{cor-dblroot-exist-mu}
  Let $f \colon K[X] \to K[X]$ be a linear bijection mapping polynomials with a $k$-fold root in $K$ to non-$k$-free polynomials.
  Then, there is a map $\mu_f \colon K \to K$ such that $x \in K$ is a $k$-fold root of~$P$ if and only if $\mu_f (x)$ is a $k$-fold root of $f(P)$.
  Moreover, if $f^{-1}$ also maps polynomials with a $k$-fold root in $K$ to non-$k$-free polynomials, then $\mu_f$ is bijective of inverse $\mu_{f^{-1}}$, and in particular~$f$ preserves the number of distinct $k$-fold roots that a given polynomial has in $K$.
\end{corollary}

\begin{proof}
  Let $x \in K$.
  The hypothesis implies that $f((X-x)^k K[X])$ is a subspace of $K[X]$ of codimension~$k$ consisting only of polynomials which are not $k$-free.
  By \Cref{cor-codimk-nonkfree}, there is a unique~$\mu_f (x) \in K$ such that $f((X-x)^k K[X]) = (X-\mu_f (x))^k K[X]$.
  The other points follow.
\end{proof}

\begin{remark}
  The hypothesis that the characteristic of $K$ is zero or greater than $k$ in \Cref{lem-kfree-to-kroot} is crucial.
  For instance, assume that $K$ is an infinite field of characteristic $p$ and that $k$ is a power of~$p$.
  Let $x \in K^{\times}$.
  Then $(X-x)^k$ does not divide $X^{2k}$ but, for any $\lambda \in K$ which is a $k$-th power, the polynomial $X^{2k} + \lambda (X-x)^k$ is not $k$-free as it equals $\leftl(X^2 + \sqrt[k]{\lambda} (X - x)\rightr)^k$.
\end{remark}

\subsection{From $k$-fold roots to $(k-1)$-fold roots}
\label{subsn-kroot-induction}

We fix a field $K$ of characteristic zero.
In this subsection, we prove \Cref{thm-sub-main}, which characterizes $K$-linear bijections $K[X] \to K[X]$ preserving $k$-fold roots up to a fixed permutation of $K$.
The proof is by induction on $k$, using the results of \Cref{sn-root-permuting} for the case $k=1$, and using the following lemma to relate the cases $k$ and $k-1$:

\begin{lemma}
  \label{lem-roots-of-det}
  Let $k \geq 2$ be an integer, let $P, Q \in K[X]$, and let $x \in K$ be such that at least one of~$P(x)$ and~$Q(x)$ is nonzero.
  Then, $x$ is a $(k-1)$-fold root of $P Q' - P' Q$ if and only if $x$ is a $k$-fold root of $\lambda P + \nu Q$ for some $(\lambda, \nu) \in K^2 \setminus \{(0,0)\}$.
\end{lemma}

\begin{proof}
  Using the equality
  $
    \leftl(
      \det\!
      \begin{pmatrix}
        A & B \\
        C & D
      \end{pmatrix}
    \rightr)'
    =
    \det\!
    \begin{pmatrix}
      A' & B' \\
      C & D
    \end{pmatrix}
    +
    \det\!
    \begin{pmatrix}
      A & B \\
      C' & D'
    \end{pmatrix}
  $
  repeatedly, we obtain:%
  \footnote{
    If we take cancellations of determinants into account, we can be more precise:
    \[
      \leftl(
        \det\!
        \begin{pmatrix}
          P & Q \\
          P' & Q'
        \end{pmatrix}
      \rightr)^{(n)}
      =
      \sum_{i=0}^{
        \leftl\lceil
          n / 2
        \rightr\rceil
      }
        C(n-i,i)
        \cdot
        \det\!
        \begin{pmatrix}
            P^{(i)} & Q^{(i)} \\
            P^{(n+1-i)} & Q^{(n+1-i)}
        \end{pmatrix}
    \]
    where~$C(n,p)$ denotes Catalan numbers.
  }
  \[
    \forall n \in \N, \quad
    (P Q' - P' Q)^{(n)}
    =
    \leftl(
      \det\!
      \begin{pmatrix}
        P & Q \\
        P' & Q'
      \end{pmatrix}
    \rightr)^{(n)}
    =
    \sum_{i=0}^n
      \binom{n}{i}
      \cdot
      \det\!
      \begin{pmatrix}
        P^{(i)} & Q^{(i)} \\
        P^{(n+1-i)} & Q^{(n+1-i)}
      \end{pmatrix}
      .
  \]
  Now, assume that there exists $(\lambda, \nu) \in K^2 \setminus \{(0,0)\}$ such that $x \in K$ is a $k$-fold root of $\lambda P + \nu Q$, i.e., that the vectors $\leftl( P(x), P'(x), \dots, P^{(k-1)}(x) \rightr)$ and $\leftl( Q(x), Q'(x), \dots, Q^{(k-1)}(x) \rightr)$ are colinear.
  Then, all the determinants
  $
    \det\!
    \begin{pmatrix}
      P^{(i)}(x) & Q^{(i)}(x) \\
      P^{(n+1-i)}(x) & Q^{(n+1-i)}(x)
    \end{pmatrix}
  $
  vanish for $0 \leq i \leq n \leq k-1$, so $x$ is a $(k-1)$-fold root of $P Q' - P' Q$.

  Conversely, assume that $x$ is a $(k-1)$-fold root of $P Q' - P' Q$.
  We show by induction on $n \in \{0, \dots, k-1\}$ that $\leftl( Q(x), Q'(x), \dots, Q^{(n)}(x) \rightr)$ and $\leftl( P(x), P'(x), \dots, P^{(n)}(x) \rightr)$ are colinear, and more precisely that:
  \[
    P(x)
    \cdot
    \leftl(
      Q(x), Q'(x), \dots, Q^{(n)}(x)
    \rightr)
    -
    Q(x)
    \cdot
    \leftl(
      P(x), P'(x), \dots, P^{(n)}(x)
    \rightr)
    = 0.
  \]
  Note that $(-Q(x), P(x)) \in K^2 \setminus \{(0,0)\}$ as $x$ is not a common root of $P$ and $Q$.
  The result for $n = 0$ is clear.
  Fix an $n \in \{1, \dots, k-1\}$, and assume that $P(x) \cdot (Q(x), Q'(x), \dots, Q^{(n-1)}(x)) -  Q(x) \cdot (P(x), P'(x), \dots, P^{(n-1)}(x)) = 0$.
  In particular, all determinants
  $
    \det\!
    \begin{pmatrix}
      P^{(i)}(x) & Q^{(i)}(x) \\
      P^{(j)}(x) & Q^{(j)}(x)
    \end{pmatrix}
  $ for $i, j \in \{0, \dots, n-1\}$ vanish.
  We have:
  \begin{align*}
    0
    =
    (P Q' - P' Q)^{(n-1)}(x)
    & =
    \sum_{i=0}^{n-1}
      \binom{n-1}{i}
      \cdot
      \det\!
      \begin{pmatrix}
        P^{(i)}(x) & Q^{(i)}(x) \\
        P^{(n-i)}(x) & Q^{(n-i)}(x)
      \end{pmatrix}
    \\
    & =
    \det\!
    \begin{pmatrix}
      P(x) & Q(x) \\
      P^{(n)}(x) & Q^{(n)}(x)
    \end{pmatrix}
  \end{align*}
  and hence $P(x) Q^{(n)}(x) - Q(x) P^{(n)}(x) = 0$.
  By induction, we have shown:
  \[
    P(x)
    \cdot
    \leftl(
      Q(x), Q'(x), \dots, Q^{(k-1)}(x)
    \rightr)
    -
    Q(x)
    \cdot
    \leftl(
      P(x), P'(x), \dots, P^{(k-1)}(x)
    \rightr)
    = 0.
  \]
  Hence, there is a pair $(\lambda, \nu) \in K^2 \setminus \{(0,0)\}$ such that $x$ is a $k$-fold root of $\lambda P + \nu Q$, namely $\lambda = -Q(x)$ and $\nu=P(x)$.
\end{proof}

We also need the following small lemma:

\begin{lemma}
  \label{lem:criterion-scalar-multiples}
  Let $P, Q \in K[X]$ be polynomials satisfying $P Q' = P' Q$, with $P \neq 0$.
  Then, there is a constant $\lambda \in K$ such that $Q = \lambda P$.
\end{lemma}

\begin{proof}
  If $Q = 0$, we can take $\lambda = 0$.
  Hence, we assume that $Q \neq 0$.
  Comparing the leading monomials of $PQ'$ and $P'Q$ shows that $P$ and $Q$ must have the same degree, so $Q = \lambda P + \tilde Q$ for some $\lambda \in K^\times$ and some $\tilde Q \in K[X]$ of degree $< \deg P$.
  The equality $PQ' = P'Q$ implies $P \tilde Q' = P' \tilde Q$.
  By the same arguments as those used for $Q$, this can only occur if $\deg \tilde Q = \deg P$, which is not true, or if $\tilde Q = 0$, which implies the claim.
\end{proof}

Let $\igl \colon K[X] \to K[X]$ be the linear map taking a polynomial to its formal antiderivative vanishing at~$0$, i.e., the linear map induced by $X^n \mapsto \frac{X^{n+1}}{n+1}$, which is well-defined as $K$ has characteristic zero.
For all polynomials $P \in K[X]$, we have $(\igl P)' = P$ and $\igl (P') = P - P(0)$.

\begin{theorem}
  \label{thm-sub-main}
  Let $k \geq 1$ be an integer.
  Let $S$ be a finite subset of $K$ and $K' \coloneqq K \setminus S$.
  Let $\mu \colon K \to K$ be a bijective map and let $f_1, f_2 \colon K[X] \to K[X]$ be two $K$-linear maps such that, for all $P \in K[X]$ and $x \in K'$:
  \begin{equation}
    \label{eqn-hyp}
    x \text{ is a } k \text{-th root of } P
    \iff
    \mu(x) \text{ is a } k \text{-th root of } f_1(P)
    \iff
    \mu^{-1}(x) \text{ is a } k \text{-th root of } f_2(P).
  \end{equation}
  Then, there are constants $a \in K^{\times}$ and $b \in K$ such that:
  \begin{enumerate}
    \item
      $\mu(x) = \frac{x-b}{a}$ for all but finitely many $x \in K$
    \item
      $
        f_1(P)(X)
        =
        f_1(1)
        P(a X + b)
      $ and $
        f_2(P)(X)
        =
        f_2(1)
        P\leftl(
          \frac{X-b}{a}
        \rightr)
      $, for all $P \in K[X]$
    \item
      $f_1 \circ f_2 = f_2 \circ f_1 = f_1(1) f_2(1) \cdot \id$.
  \end{enumerate}
\end{theorem}

\begin{proof}
  We reason by induction on~$k$.
  The case $k=1$ is \Cref{cor-mu-linear}.
  Let $k \geq 2$, and assume that the theorem holds for~$k-1$.

  Let $P \in K[X]$.
  Note that
  $
    P
    =
    (\igl P)'
    =
    \det\!
    \begin{pmatrix}
      1 & \igl P \\
      1' & (\igl P)'
    \end{pmatrix}
  $.

  Let $x \in K'$, and assume that~$x$ does not belong to the finite set
  $
    S'_2
    \coloneqq
    \leftl\{
      \mu(y)
      \verti
      y \text{ root of } f_2(1)
    \rightr\}
  $.
  Then:
  \begin{align*}
    &
    \Bigl[
      x \text{ is a $(k-1)$-fold root of } P
    \Bigr]
    \\
    \Leftrightarrow
    &
    \leftl[
      \exists (\lambda, \nu) \in K^2 \setminus \{(0,0)\},
      \,\,
      x \text{ is a $k$-fold root of } \lambda \cdot 1 + \nu \cdot \igl P
    \rightr]
    &
    \text{by \Cref{lem-roots-of-det}}
    \\
    \Leftrightarrow
    &
    \leftl[
      \exists (\lambda, \nu) \in K^2 \setminus \{(0,0)\},
      \,\,
      \mu^{-1}(x) \text{ is a $k$-fold root of } \lambda \cdot f_2(1) + \nu \cdot f_2(\igl P)
    \rightr]
    &
    \text{by}
    \begin{cases}
      x \in K' \\
      \text{\Cref{eqn-hyp}}
    \end{cases}
    \\
    \Leftrightarrow
    &
    \Bigl[
      \mu^{-1}(x) \text{ is a $(k-1)$-fold root of } f_2(1) f_2(\igl P)' - f_2(1)' f_2(\igl P)
    \Bigr]
    &
    \text{by}
    \begin{cases}
      x \not\in S'_2 \\
      \text{\Cref{lem-roots-of-det}}
    \end{cases}
  \end{align*}
  Define the $K$-linear map $\tilde{f_2} \colon P \mapsto f_2(1) f_2(\igl P)' - f_2(1)' f_2(\igl P)$.
  We have shown:
  \[
    \forall x \in K' \setminus S'_2, \quad
    [
      x \text{ is a $(k-1)$-fold root of } P
    ]
    \iff
    [
      \mu^{-1}(x) \text{ is a $(k-1)$-fold root of } \tilde{f_2}(P)
    ]
  \]
  Symmetrically, the $K$-linear map $\tilde{f_1} \colon P \mapsto f_1(1) f_1(\igl P)' - f_1(1)' f_1(\igl P)$ and the finite set $S'_1 \coloneqq \leftl\{ \mu(y) \verti y \text{ root of } f_1(1) \rightr\}$ satisfy:
  \[
    \forall x \in K' \setminus S'_1, \quad
    [
      x \text{ is a $(k-1)$-fold root of } P
    ]
    \iff
    [
      \mu(x) \text{ is a $(k-1)$-fold root of } \tilde{f_1}(P)
    ]
  \]
  Define the finite set $S' = S \cup S'_1 \cup S'_2$.
  By the induction hypothesis, applied to $(S', \mu, \tilde{f_1}, \tilde{f_2})$, there are constants $a \in K^{\times}$ and $b \in K$ such that $\mu(x) = \frac{x-b}{a}$ for all but finitely many~$x \in K$ and such that, for all $P \in K[X]$:
  \[
    \tilde{f_1}(P)(X)
    =
    \tilde{f_1}(1)
    P(a X + b)
    \andd
    \tilde{f_2}(P)(X)
    =
    \tilde{f_2}(1)
    P\leftl(
      \frac{X-b}{a}
    \rightr).
  \]
  The first equation implies:
  \begin{equation}
    \label{eqn-tilde-f-new}
    f_1(1) f_1(P)' - f_1(1)' f_1(P) = \tilde{f_1}(1) P'(a X + b)
  \end{equation}
  To simplify notations, we replace $f_1$ by the map
  $
    P
    \mapsto
    f_1\leftl(
      P \circ \frac{X-b}{a}
    \rightr)
  $, so that $a=1$ and $b=0$.
  Our goal then becomes to prove that $f_1(P) = f_1(1) P$ for all polynomials $P \in K[X]$ (the case of $f_2$ is symmetrical, and the third point follows from the second point).

  Using Euclidean division, write $\tilde{f_1}(1) = Q \cdot f_1(1)^2 + R$ for polynomials $Q, R \in K[X]$ satisfying $\deg R < 2 \deg f_1(1)$, and define the $K$-linear map $h_Q \colon K[X] \to K[X]$ by:
  \[
    h_Q (P) 
    \coloneqq
    f_1(1)
    \cdot
    [Q P - Q' (\igl P) + Q'' (\igl^2 P) - \dots]
  \]
  which is a finite sum as $Q$ is a polynomial.
  Note that:
  \begin{align*}
    f_1(1) h_Q (P)' - f_1(1)' h_Q (P)
    = 
    & f_1(1)^2
    \cdot
    [
        Q P'
        + Q' P
        - Q' P
        - Q'' (\igl P)
        + Q'' (\igl P)
        + Q''' (\igl^2 P)
        - \dots
    ]
    \\
     &
    + f_1(1) f_1(1)' 
    \cdot[Q P - Q' (\igl P) + Q'' (\igl^2 P) - \dots]
    \\
    &
    - f_1(1)' f_1(1) 
    \cdot[Q P - Q' (\igl P) + Q'' (\igl^2 P) - \dots]
    \\
    =&
    f_1(1)^2 Q P'.
  \end{align*}
  Now define $g_1 \coloneqq f_1 - h_Q$.
  For all polynomials $P \in K[X]$, we have:
  \begin{equation}
    \label{eqn-big-G}
    \begin{gathered}
      f_1(1) g_1(P)' - f_1(1)' g_1(P) = R P'
      \\
      \deg R < 2 \deg f_1(1).
    \end{gathered}
  \end{equation}

  We prove by contradiction that $R=0$.
  Assume that~$R$ is nonzero.
  Then, $2 \deg f_1(1) > \deg R$ implies that $f_1(1)$ is non-constant.
  Let $P$ be any polynomial of degree $2 \deg f_1(1) - \deg R$ (which is a positive integer).
  Using \Cref{eqn-big-G}, we get:
  \[
    \deg\Bigl(
      f_1(1) g_1(P)' - f_1(1)' g_1(P)
    \Bigr)
    = \deg(R P')
    = 2 \deg f_1(1) - 1.
  \]
  We can also compute the degree directly:
  \[
    \deg\Bigl(
      f_1(1) g_1(P)' - f_1(1)' g_1(P)
    \Bigr)
    \leq
    \deg f_1(1) + \deg g_1(P) - 1
  \]
  with equality if and only if $\deg f_1(1) \neq \deg g_1(P)$.
  But this means that $\deg f_1(1)$ and $\deg g_1(P)$ are equal if and only if they are not equal, which is absurd.
  We have shown that $R=0$.

  For all $P \in K[X]$, \Cref{eqn-big-G} then becomes
  $
    f_1(1) g_1(P)' = f_1(1)' g_1(P)
  $, so by \Cref{lem:criterion-scalar-multiples} there is a linear form $u \colon K[X] \to K$, mapping a polynomial~$P$ to the constant $u(P) = \frac{g_1(P)}{f_1(1)}$. 
  Unfolding the definitions of $g_1$ and $h_Q$, the equality $g_1(P) = u(P) f_1(1)$ rewrites as:
  \begin{equation}
    \label{eqn-g1-linform}
    f_1(P) = f_1(1) 
    \cdot[Q P - Q' (\igl P) + Q'' (\igl^2 P) - \dots + u(P)].
  \end{equation}

  Let $\tilde{S}$ be the finite set obtained by adding the roots of $f_1(1)$ to $S$.
  Let $x$ be an element of $K \setminus \tilde{S}$, and let $\ell \geq k$.
  Since $x \not\in S$, $x$ is a ($k$-fold) root of the polynomial
  \[
    f_1\leftl(
      (X-x)^\ell
    \rightr)
    =
    f_1(1)
    \cdot
    \leftl[
      \sum_{i \geq 0}
        \ell!
        Q^{(i)}
        \frac{
          (X-x)^{\ell+i}
          -
          \sum_{j=0}^{i-1}
            \binom{\ell+i}{j}
            X^j (-x)^{\ell+i-j}
        }{
          (\ell+i)!
        }
      +
      \sum_{i=0}^\ell
        \binom{\ell}{i}
        (-x)^{\ell-i}
        u(X^i)
    \rightr].
  \]
  Evaluating at $x$ and dividing by $f_1(1)(x)$ (which we may do as $x \not\in \tilde{S}$), we obtain:
  \begin{align*}
    0
    & =
    \sum_{i \geq 0}
      -
      \ell!
      Q^{(i)}(x)
      \sum_{j=0}^{i-1}
        \binom{\ell+i}{j}
        \frac{
          (-1)^j
          (-x)^{\ell+i}
        }{
          (\ell+i)!
        }
    +
    \sum_{i=0}^\ell
      \binom{\ell}{i}
      (-x)^{\ell-i}
      u(X^i)
    \\
    & =
    -\ell!
    \sum_{i \geq 1}
      Q^{(i)}(x)
      \sum_{j=0}^{i-1}
        \frac
          {(-1)^j}
          {j!(\ell+i-j)!}
        (-x)^{\ell+i}
    +
    \sum_{i=0}^\ell
      \binom{\ell}{i}
      (-x)^{\ell-i}
      u(X^i).
  \end{align*}
  This equality holds for all $x \in K \setminus \tilde{S}$.
  Since $K \setminus \tilde{S}$ is infinite, the corresponding polynomial in $x$ equals the zero polynomial.
  In particular, $u(1) = u(X) = \dots = u(X^{\ell-1}) = u(X^\ell) = 0$ (coefficients in front of $x^\ell, x^{\ell-1}, \dots, x^0$).
  Since this holds for all $\ell \geq k$, the linear form $u$ is identically zero.
  Thus, we have the following equality of polynomials for all $\ell \geq k$:
  \begin{equation}
    \label{eqn:q-zero}
    0 
    =
    \sum_{i \geq 1}
      Q^{(i)}
      \sum_{j=0}^{i-1}
        \frac{
          (-1)^j
        }{
          j!(\ell+i-j)!
        }
        (-X)^{\ell+i}
        .
  \end{equation}
  For any fixed $i \geq 1$, the formula $\binom{\ell+i}{j} = \binom{\ell+i-1}{j} + \binom{\ell+i-1}{j-1}$ lets us simplify the inner (telescoping) sum
  \[
    \sum_{j=0}^{i-1}
      \frac{(-1)^j}{j!(\ell + i - j)!}
    =
    \frac{1}{(\ell+i)!}
    \sum_{j=0}^{i-1}
      (-1)^j \binom{\ell+i}j
    =
    \frac{1}{(\ell + i)!}
    (-1)^i \binom{\ell+i-1}{i-1}
    =
    \frac{(-1)^i}{\ell! (\ell + i)(i-1)!}
  \]
  so the right-hand side of \Cref{eqn:q-zero} rewrites as
  \[
    \sum_{i \geq 1}
      Q^{(i)}
      \sum_{j=0}^{i-1}
        \frac{
          (-1)^j
        }{
          j!(\ell+i-j)!
        }
        (-X)^{\ell+i}
    =
    \sum_{i \geq 1}
      \frac{Q^{(i)} (-1)^i}{\ell! (\ell + i)(i-1)!}
    (-X)^{\ell + i}
    =
    \frac{(-X)^\ell}{\ell!}
    \sum_{i \geq 1}
      \frac{Q^{(i)} X^i}{(\ell + i)(i-1)!}.
  \]

  Assume that~$Q$ is non-constant, and let $a_d X^d$ be the leading monomial of $Q$ (with $d \geq 1$, $a_d \neq 0$).
  The coefficient in front of~$X^{\ell + d}$ in the right-hand side of \Cref{eqn:q-zero} is then
  \[
    \frac{(-1)^\ell}{\ell!}
    \sum_{1 \leq i \leq d}
      \frac{d!}{(d-i)! (\ell+i) (i-1)!} \cdot a_d
  \]
  which is nonzero ($a_d$ is nonzero, and the rational number by which it is multiplied is nonzero, as a nonempty sum of nonzero rational numbers of same sign), contradicting \Cref{eqn:q-zero}.
  
  
  We have shown by contradiction that $Q$ is constant.
  Letting $P=1$ in \Cref{eqn-g1-linform} shows that $Q=1$.
  Finally, \Cref{eqn-g1-linform} rewrites as $f_1(P) = f_1(1) P$.
\end{proof}

\subsection{Proof of \Cref{thm-kfree-kroot}}
\label{subsn-proof-kfree-kroot}

Let $f \in \Aut_{K \Vect}^{k \free} (K[X])$ or $f \in \Aut_{K \Vect}^{k \rroot} (K[X])$.
By \Cref{cor-dblroot-exist-mu}, there is a bijective map $\mu \colon K \to K$ such that
$
  f\leftl(
    (X-x)^k K[X]
  \rightr)
  =
  \bigl(
    X-\mu(x)
  \bigr)^k K[X]
$ for all $x \in K$.
By \Cref{thm-sub-main}, applied to $S=\emptyset$, $f_1=f$ and $f_2=f^{-1}$, there are constants $a \in K^{\times}$ and $b \in K$ such that $f(P) = f(1) \cdot P \circ (a X + b)$ for all $P \in K[X]$.
Since $f$ is surjective, $f(1)$ is a nonzero constant $c \in K^{\times}$.
This concludes the proof of \Cref{thm-kfree-kroot}.

\section{Symmetries of the set of polynomials with a root in a number field}
\label{sn-root-nf}

Fix a number field $K$ (or any Hilbertian field, such as $\Q^{\text{ab}}$ or any function field in at least one variable over a field) and a finite (separable) extension $L|K$.
In this section, we prove the following theorem:

\begin{theorem}
  \label{thm-root-nf}
  Let $f \colon L[X] \to L[X]$ be a $K$-linear bijection preserving the set of polynomials with a root in $L$.
  Then, there are constants $a, c \in L^{\times}$, $b \in L$, and an automorphism $\sigma \in \Aut(L|K)$ such that
  \[
    \forall P \in L[X], \quad
    f(P) = c \cdot  \sigma(P) \circ (a X + b).
  \]
  In other words, $\Aut_{K \Vect}^{\textnormal{root}} (L[X]) = L^{\times} \rtimes \Aut_{K \alg} (L[X])$.
\end{theorem}

First, we prove the following lemma using Hilbert's irreducibility theorem:

\begin{lemma}
  \label{lem-hit}
  Let $P, Q \in L[X]$ be coprime polynomials, and assume that for all $\lambda \in K$, the polynomial~$P + \lambda Q$ has a root in $L$.
  Then, both $P$ and $Q$ have degree at most $1$.
\end{lemma}

\begin{proof}
  Let $f = \frac PQ \in L(X)$ and $D = \max(\deg P, \, \deg Q)$.
  Then, $f$ defines a morphism $\mathbb{P}^1_L \to \mathbb{P}^1_L$ of degree $D$ and, by hypothesis, all of $K$ is contained in $f(\mathbb{P}^1(L))$.
  Assume that $D > 1$, so that~$f(\mathbb{P}^1(L))$ is a thin set, and its complement is a Hilbertian subset of~$\mathbb{P}^1(L)$ containing no element of~$K$.
  By \cite[Corollary~12.2.3]{fieldar}, this Hilbertian set contains a Hilbertian subset of $\mathbb{P}^1(K)$.
  Thus, the empty set (in~$\mathbb{P}^1(K)$) is Hilbertian, contradicting the fact that $K$ is Hilbertian.
  We have shown that~$D=1$.
\end{proof}


As a consequence, a version of \Cref{cor-codimk-nonkfree} holds also for $k=1$ in the case of Hilbertian fields:

\begin{corollary}
  \label{cor-codim1-root-nf}
  Linear $K$-subspaces of $L[X]$ of codimension $[L:K]$ containing only polynomials with a root in~$L$ are exactly subspaces of the form $(X-x) L[X]$ for some $x \in L$.
\end{corollary}

\begin{proof}
  Let $V$ be such a linear subspace.
  Since $V$ has codimension $[L:K]$, which is the dimension of the space of polynomials of degree $\leq 0$, and since nonzero polynomials of degree $\leq 0$ do not have a root in $L$, we have $V \oplus L = L[X]$.
  The space $V'$ of polynomials of degree $\leq 1$ in~$V$ has dimension exactly $[L:K]$, and nonzero elements of $V'$ are polynomials with degree exactly $1$.

  Let $P \in V' \setminus \{0\}$ and $Q \in V$ of degree $\geq 2$, and assume that~$P$ does not divide~$Q$.
  Then, $Q$ is coprime to $P$ since $\deg P = 1$.
  Moreover, $P + \lambda Q$ belongs to $V$ and hence has a root in $L$, for all $\lambda \in K$.
  But then, \Cref{lem-hit} implies $\deg Q \leq 1$, which contradicts the hypothesis.
  Therefore, every element of $V' \setminus \{0\}$ divides every element of $V$ of degree~$\geq 2$.

  Fix any polynomial $Q_0 \in V$ of degree $2$, which is split since it has a root in~$L$, and let $x_1, x_2$ be its roots in $L$ (which may be equal).
  As seen previously, any $P \in V' \setminus \{0\}$ divides $Q_0$, which means that $V'$ is covered by the two vector spaces $L \cdot (X-x_1)$ and $L \cdot (X-x_2)$.
  As the field $K$ is infinite, this implies $V' \subseteq L \cdot (X-x_i)$ for some $i \in \{1,2\}$.
  Since these spaces both have dimension $[L:K]$, we have the equality $V' = L \cdot (X-x_i)$.
  To show that all polynomials in $V$ are divisible by $X-x_i$, remark that they either belong to $V' = L \cdot (X-x_i)$, or they are of degree $\geq 2$ and are then divisible by $X-x_i \in V' \setminus \{0\}$.
  We have proved that $V \subseteq (X-x_i)L[X]$.
  Since these spaces have the same codimension $[L:K]$, this is an equality.
\end{proof}

Finally, we prove \Cref{thm-root-nf}:

\begin{proof}[Proof of \Cref{thm-root-nf}]
  Let $f \in \Aut_{K \Vect}^{\textnormal{root}} (L[X])$.
  For all $x \in L$, the image of $(X-x)L[X]$ under~$f$ is a $K$-subspace of $L[X]$ of codimension $[L:K]$ consisting only of polynomials with a root in~$L$.
  By \Cref{cor-codim1-root-nf}, this implies that there is a map $\mu \colon L \to L$ such that
  $
    f\!\bigl(
      (X-x)L[X]
    \bigr)
    =
    \bigl(
      X-\mu(x)
    \bigr)L[X]
  $ for all $x \in K$.
  Doing the same for $f^{-1}$ shows that $\mu$ is bijective.
  Let $S$ be the finite set
  $
    \leftl\{
      \mu^{-1}(x)
      \verti
      f(1)(x)=0
    \rightr\}
    \cup
    \leftl\{
      \mu(x)
      \verti
      f^{-1}(1)(x)=0
    \rightr\}
  $
  and $L' = L\setminus S$.
  Applying \Cref{thm-root-pres} to the triples $(S, \mu|_{L'}, f)$ and $(S, \mu^{-1}|_{L'}, f^{-1})$ shows that there are polynomials $Y_1, Y_2 \in L[X]$ and automorphisms $\sigma, \tau \in \Aut(L|K)$ such that:
  \[
    \forall P \in L[X],\quad
    f(P) = f(1) \cdot (\sigma(P) \circ Y_1)
    \andd
    f^{-1}(P) = f^{-1}(1) \cdot (\tau(P) \circ Y_2)
  \]
  and such that
  $
    \leftl\lbrace
      \begin{array}{ccl}
        \sigma & \text{coincides with} & Y_1 \circ \mu \\
        \tau & \text{coincides with} & Y_2 \circ \mu^{-1}
      \end{array}
    \rightr.
  $ on $L'$.
  Since $f$ is surjective, $f(1)$ is a nonzero constant~$c \in L^{\times}$, and $f^{-1}(1) = \frac{1}{\sigma^{-1}(c)}$.
  At all but finitely many points, we have the equality $\sigma \circ \tau = \sigma \circ Y_2 \circ \mu^{-1} = \sigma(Y_2) \circ \sigma \circ \mu^{-1} = \sigma(Y_2) \circ Y_1$.
  But then the polynomial $\sigma(Y_2) \circ Y_1 \in L[X]$ coincides with the identity $X \in L[X]$ when evaluated at all but finitely many elements of $K$ (forming an infinite subset of $L$).
  This implies the algebraic equality $\sigma(Y_2) \circ Y_1 = X$ and thus $\sigma \circ \tau = \id$, i.e., $\tau = \sigma^{-1}$.
  Moreover, the equality $\sigma(Y_2) \circ Y_1 = X$ in $L[X]$ implies that both $Y_1$ and $Y_2$ are polynomials of degree~$1$.
  If we write $Y_1 = a X + b$, then
  $Y_2 = \sigma^{-1}\leftl(\frac{X-b}{a}\rightr) = \frac{X-\sigma^{-1}(b)}{\sigma^{-1}(a)}$.
\end{proof}

\section{Symmetries of the set of polynomials with a real root}
\label{sn-root-real}

The goal of this section is to prove the following theorem:
\begin{theorem}
  \label{thm-root-real}
  Let $f$ be a $\R$-linear bijective map $\R[X] \to \R[X]$ preserving the set of polynomials with a real root.
  Then, there are real numbers $a, c \in \R^{\times}$ and $b \in \R$ such that:
  \[
    \forall P \in \R[X], \quad
    f(P) = c \cdot P \circ (a X + b).
  \]
  In other words, $\Aut_{\R \Vect}^{\textnormal{root}} (\R[X]) = \R^{\times} \rtimes \Aut_{\R \alg} (\R[X])$.
\end{theorem}

In \Cref{subsn-vpq-sp}, we prove lemmas which imply that maps in $\Aut_{\R \Vect}^{\textnormal{root}} (\R[X])$ preserve the set of polynomials with a given degree, and in \Cref{subsn-quad-dblroot} we show that these maps preserve the set of polynomials which have a real double root.
In \Cref{subsn-proof-root-real}, we prove \Cref{thm-root-real} by combining the lemmas from the first two subsections with \Cref{thm-sub-main}.

\begin{remark}
  Without the bijectivity hypothesis, there are additional maps even if we require $f(1)=~1$.
  For instance, the map $f \colon P \mapsto P \circ Q$, where $Q \in \R[X]$ is a non-constant polynomial whose derivative has no real roots of odd multiplicity (e.g. $Q=X^3$), preserves the set of polynomials with a real root.
\end{remark}

\begin{remark}
  The naive variant of \Cref{cor-codimk-nonkfree}/\Cref{cor-codim1-root-nf} does not hold for real roots: there are many linear subspaces of $\R[X]$ of codimension $1$ consisting only of polynomials with a real root, and which are not of the form $(X-x)\R[X]$.
  Here is a way to construct such subspaces: let $\mu$ be a nonzero positive measure on $\R$ with all moments finite (for example, such that its support $\mathrm{Supp}(\mu)$ is compact), and let~$L$ be the subspace (of codimension~$1$) of polynomials $P \in \R[X]$ such that~$\int_{\R} P \, \mathrm{d} \mu = 0$.
  If $P \in L$, then either $P$ vanishes everywhere on $\mathrm{Supp}(\mu)$ in which case~$P$ has a real root, or $P$ must take values of both signs on $\mathrm{Supp}(\mu)$ for the integral to vanish, so~$P$ has a zero in the convex hull of $\mathrm{Supp}(\mu)$ by the intermediate value theorem.
  Thus, all polynomials~$P \in L$ have a real root.
  For instance, if~$\mu$ is the sum of two Dirac measures concentrated respectively at~$-1$ and~$1$, then~$L$ is the space of polynomials~$P$ satisfying $P(1)+P(-1)=0$, which all have a real root; that space is spanned by the polynomials $X, X^2-1, X^3, X^4-1, X^5, X^6-1, \dots$
\end{remark}

\subsection{The sets $V(P,Q)$ and $S(P)$}
\label{subsn-vpq-sp}

In this subsection, we describe properties of the following sets:
\begin{itemize}
  \item
    for $P, Q \in \R[X]$, we let
    $
      V(P,Q)
      \coloneqq
      \leftl\{
        \lambda \in \R
        \verti
        P - \lambda Q \text{ has a real root}
      \rightr\}
    $
  \item
    for $P \in \R[X]$, we let
    $
      S(P)
      \coloneqq
      \leftl\{
        Q \in \R[X]
        \verti
        V(Q,P) \text{ is bounded}
      \rightr\}
    $
\end{itemize}
In \Cref{lem-even-degree}, we characterize the set of polynomials of degree $\leq n$ for any even $n$ in terms of the sets~$S(P)$.
This is later used to prove that any map $f \in \Aut_{\R \Vect}^{\textnormal{root}} (\R[X])$ preserves the set of polynomials of a given degree.

\begin{lemma}
  \label{lem-descr-VPQ}
  Let $P,Q \in \R[X]$.
  If $P$ and $Q$ have a common real root, then $V(P,Q) = \R$.
  Otherwise, $V(P,Q)$ is the image of $\frac PQ$.
\end{lemma}

\begin{proof}
  If $P(x)=Q(x)=0$, then $P -\lambda Q$ has the real root $x$ for all $\lambda \in \R$, which shows $V(P,Q) = \R$.
  Assume now that~$P$ and~$Q$ have no common real root.
  Let $y$ be in the image of $\frac PQ$, i.e., $P(x) = y Q(x)$ for some real number $x$ such that $Q(x) \neq 0$.
  Then, $P - y Q$ has the real root $x$.
  Conversely, if $P - y Q$ has the real root $x$, then $P(x)=y Q(x)$.
  If $Q(x)=0$, then $P(x)=0$ contradicts the fact that $P$ and~$Q$ have no common real root.
  Therefore, $Q(x) \neq 0$ and $\frac {P(x)}{Q(x)} = y$.
\end{proof}




\begin{lemma}
  \label{lem-VQP-bounded}
  Let $P \in \R[X]$.
  If $P$ has a real root, then $S(P)$ is empty.
  Otherwise, $S(P)$ is the space of polynomials of degree $\leq \deg P$.
\end{lemma}

\begin{proof}
  By \Cref{lem-descr-VPQ}, the set $S(P)$ is the set of polynomials having no common real root with $P$, and such that the function $\frac Q P$ is bounded.
  If $P$ has a real root $x$, then for every polynomial~$Q$ sharing no real root with $P$, the function $\frac Q P$ has asymptotes at $x$ and hence takes unbounded values; thus, $S(P)=\emptyset$.
  Now, assume that~$P$ has no real root.
  Then, $S(P)$ is the set of polynomials $Q$ such that the function $\frac Q P$ (which has no asymptotes at any real number) is bounded on $\R$; this amounts to the fact that it does not become infinite at $\pm \infty$, i.e., that $\deg Q \leq \deg P$.
\end{proof}

\begin{lemma}
  \label{lem-even-degree}
  Let $n$ be an even integer and $V$ be a subspace of $\R[X]$.
  Then, $V$ is the space of polynomials of degree at most $n$ if and only if $\dim(V) = n+1$ and $V = S(P_0)$ for some $P_0 \in V$.
\end{lemma}

\begin{proof}
  If $V$ is the space of polynomials of degree at most $n$, which has dimension $n+1$, then the polynomial $P_0 = X^n + 1 \in V$ has no real root and has degree $n$, and hence $S(P_0) = V$ by \Cref{lem-VQP-bounded}.
  Conversely, assume that $\dim(V) = n+1$ and that $V= S(P_0)$ for some $P_0 \in V$.
  By \Cref{lem-VQP-bounded}, the polynomial $P_0$ has no real root and $V$ is the space of polynomials of degree $\leq \deg P_0$.
  Since $\dim V = n+1$, we have $\deg P_0 = n$.
\end{proof}

\subsection{Quadratic polynomials and double roots}
\label{subsn-quad-dblroot}

In this subsection, we characterize polynomials having a given $x \in \R$ as a double root using only the ``additive'' structure of polynomials and the property of having a real root (assuming we can ``recognize'' polynomials of degrees $0$, $1$ and $2$, and nonnegative polynomials).

\begin{lemma}
  \label{lem-second-degree}
  Let $x \in \R$, and let $P$ be a polynomial of degree $2$ such that $P + b (X-x) + c$ has a real root if and only if $b^2 \geq 4c$.
  Then, $P=(X-x)^2$.
\end{lemma}

\begin{proof}
  Write $P = \alpha (X-x)^2 + \beta (X-x) + \gamma$ with $\alpha \neq 0$.
  The hypothesis says that
  \[ 
    [ (\beta + b)^2 \geq 4 \alpha (\gamma + c) ]
    \iff
    [ b^2 \geq 4c ],
  \]
  i.e., the affine transformation
  $
  \leftl\{
    \begin{matrix}
      b & \mapsto & b + \beta \\
      c & \mapsto & \alpha c + \alpha \gamma
    \end{matrix}
  \rightr.$
  preserves the set
  $\leftl\{
    (b,c)
    \verti
    b^2 \geq 4 c
  \rightr\}$,
  and hence preserves its boundary, the parabola $b^2 = 4c$.
  This implies $(b+\beta)^2 = \alpha b^2 + 4 \alpha \gamma$ for all~$b \in \R$. 
  Equating the coefficients of the corresponding polynomials in $b$, we obtain $1 = \alpha$, $2 \beta = 0$ and $\beta^2 = 4 \alpha \gamma$, from which follows that $(\alpha, \beta, \gamma) = (1,0,0)$ and thus $P=(X-x)^2$.
\end{proof}



\begin{lemma}
  \label{lem-dblroot-evenposlead}
  Let $P$ be a polynomial of even degree with positive leading coefficient, and $x \in \R$.
  Then,~$x$ is a double root of $P$ if and only if there is a $\lambda_0 \in \R$ such that the polynomial $P + \lambda (X-x)^2$ has image $\R^+$ for any $\lambda > \lambda_0$.
\end{lemma}

\begin{proof}
  First assume that $x$ is a double root of $P$, and write $P=(X-x)^2 \tilde{P}$.
  Then, $P + \lambda (X-x)^2 = (\tilde{P} + \lambda)(X-x)^2$.
  Since $\tilde{P}$ is a polynomial of even degree and positive leading coefficient, it has a global minimum $m \in \R$.
  Taking $\lambda_0 = -m$, we have $\tilde{P}(t) + \lambda \geq 0$ for all $\lambda > \lambda_0$ and all~$t \in \R$, and then the polynomial $P + \lambda (X-x)^2$ has image $\R^+$ (it is nonzero, everywhere nonnegative, and vanishes at $x$).

  Conversely, assume that there is a $\lambda_0 \in \R$ such that the polynomial $P + \lambda (X-x)^2$ has image~$\R^+$ for any $\lambda > \lambda_0$.
  In particular, for any $\lambda > \lambda_0$, the polynomial $P + \lambda (X-x)^2$ has a root and since that polynomial is nonnegative, that root is of even multiplicity --- in particular, it has multiplicity at least $2$.
  Then, \Cref{lem-kfree-to-kroot} (with $k=2$) implies that $(X-x)^2$ divides $P$.
\end{proof}

\begin{lemma}
  \label{lem-diff-even-pol}
  Let $P \in \R[X]$ and $x \in \R$.
  Then, $x$ is a double root of $P$ if and only if there are polynomials $P_1, P_2 \in \R[X]$ of even degree and positive leading coefficient such that $P = P_1 - P_2$ and such that $x$ is a double root of both $P_1$ and $P_2$.
\end{lemma}

\begin{proof}
  If $P=P_1-P_2$ and both $P_1$ and $P_2$ have $x$ as a double root, then $x$ is a double root of~$P$.
  Conversely, assume that $x$ is a double root of $P$.
  If $P = 0$, take $P_1 = P_2 = (X-x)^2$.
  If $P$ has even degree, take $(P_1, P_2) = (P, 0)$ if~$P$ has positive leading coefficient and $(P_1, P_2) = (0, -P)$ otherwise.
  If~$P$ has odd degree~$n$, take $P_1 = (X-x)^{n+1} + P$ and $P_2 = (X-x)^{n+1}$.
\end{proof}

\subsection{Proof of \Cref{thm-root-real}}
\label{subsn-proof-root-real}

Let $f \in \Aut_{\R \Vect}^{\textnormal{root}} (\R[X])$.
Note that, for all polynomials $P, Q \in \R[X]$:
\begin{equation}
  \label{eqn-real-vpq}
  V(f(P), f(Q)) = V(P, Q)
  \andd
  S(f(P)) = f(S(P)).
\end{equation}
By \Cref{lem-even-degree}, this implies that $f$ preserves the set of polynomials of degree $\leq n$ when~$n$ is even.
In particular, the case $n=0$ implies that $f$ maps constant polynomials to constant polynomials.
We reduce to the case $f(1)=1$ by replacing $f$ by $\frac f {f(1)}$.

Let $P \in \R[X]$.
Then, \Cref{eqn-real-vpq} implies that $V(f(P),1)=V(P,1)$, which by \Cref{lem-descr-VPQ} means that the polynomials $P$ and $f(P)$ have the same image.
In particular:
\begin{itemize}
  \item
    $f$ maps polynomials of odd degree (surjective polynomials) to polynomials of odd degree;
  \item
    $f$ maps polynomials of even degree to polynomials of even degree, with a leading coefficient of the same sign, and with the same global minimum/maximum.
    In particular, $f$ preserves the set of nonnegative polynomials.
\end{itemize}
The map $f$ then preserves the set of polynomials of degree $n$:
\begin{itemize}
  \item
    when $n$ is odd, because they are the polynomials whose degree is odd, $\leq n+1$ and not $\leq n-1$.
  \item
    when $n$ is even, because they are the polynomials whose degree is even, $\leq n$ and not $\leq n-2$.
\end{itemize}
In particular, $f(X)$ is a polynomial of degree $1$, which we write as $a X + b$ with $a \in \R^{\times}$ and $b \in \R$.
We reduce to the case~$f(X)=X$ by replacing $f$ by $P \mapsto f\leftl( P \circ \leftl( \frac{X-b}{a} \rightr) \rightr)$.

Let $u \in \R$.
By \Cref{lem-second-degree}, $f$ maps $(X-u)^2$ to $(X-u)^2$.
Then, \Cref{lem-dblroot-evenposlead} implies that~$f$ preserves the set of polynomials with even degree and positive leading coefficient having $u$ as a double root.
By \Cref{lem-diff-even-pol}, this implies that $f$ preserves the set $(X-u)^2 \R[X]$ of all polynomials having $u$ as a double root.
We conclude using \Cref{thm-sub-main} with $k=2$, $S=\emptyset$, $\mu = \id_\R$ and $(f_1, f_2) = (f, f^{-1})$.

\bibliographystyle{alphaurl}
\bibliography{ref.bib}
\end{document}